\documentclass[11pt]{amsart}
\usepackage{amsopn, amssymb, amscd, amsmath, amsthm}
\usepackage{graphicx, graphics, epsfig}

\newcommand{\nc}{\newcommand}

\nc{\fg}{\mathfrak{f}}  \nc{\vg}{\mathfrak{v}} \nc{\wg}{\mathfrak{w}} \nc{\zg}{\mathfrak{z}} \nc{\ngo}{\mathfrak{n}} \nc{\kg}{\mathfrak{k}} \nc{\mg}{\mathfrak{m}} \nc{\bg}{\mathfrak{b}} \nc{\ggo}{\mathfrak{g}} \nc{\ggob}{\overline{\mathfrak{g}}} \nc{\sog}{\mathfrak{so}} \nc{\sug}{\mathfrak{su}} \nc{\spg}{\mathfrak{sp}} \nc{\slg}{\mathfrak{sl}} \nc{\glg}{\mathfrak{gl}} \nc{\cg}{\mathfrak{c}} \nc{\rg}{\mathfrak{r}}  \nc{\hg}{\mathfrak{h}} \nc{\tgo}{\mathfrak{t}} \nc{\ug}{\mathfrak{u}} \nc{\dg}{\mathfrak{d}} \nc{\ag}{\mathfrak{a}} \nc{\pg}{\mathfrak{p}} \nc{\sg}{\mathfrak{s}} \nc{\affg}{\mathfrak{aff}} \nc{\qg}{\mathfrak{q}}
\nc{\Xg}{\mathfrak{X}}

\nc{\pca}{\mathcal{P}} \nc{\nca}{\mathcal{N}} \nc{\lca}{\mathcal{L}} \nc{\oca}{\mathcal{O}} \nc{\mca}{\mathcal{M}} \nc{\tca}{\mathcal{T}} \nc{\aca}{\mathcal{A}} \nc{\cca}{\mathcal{C}} \nc{\gca}{\mathcal{G}} \nc{\sca}{\mathcal{S}} \nc{\hca}{\mathcal{H}} \nc{\bca}{\mathcal{B}} \nc{\dca}{\mathcal{D}}

\nc{\vp}{\varphi} \nc{\ddt}{\tfrac{{\rm d}}{{\rm d}t}} \nc{\dds}{\tfrac{{\rm d}}{{\rm d}s}} \nc{\ddtbig}{\frac{{\rm d}}{{\rm d}t}} \nc{\dd}{{\rm d}}
\nc{\dpar}{\tfrac{\partial}{\partial t}} \nc{\im}{\mathtt{i}}

\nc{\SO}{\mathrm{SO}} \nc{\Spe}{\mathrm{Sp}} \nc{\Sl}{\mathrm{SL}}
\nc{\SU}{\mathrm{SU}} \nc{\Or}{\mathrm{O}} \nc{\U}{\mathrm{U}} \nc{\Gl}{\mathrm{GL}}
\nc{\Se}{\mathrm{S}} \nc{\Cl}{\mathrm{Cl}} \nc{\Spein}{\mathrm{Spin}}
\nc{\Pin}{\mathrm{Pin}} \nc{\G}{\mathrm{GL}_n(\RR)} \nc{\g}{\mathfrak{gl}_n(\RR)}

\nc{\RR}{{\mathbb R}} \nc{\HH}{{\mathbb H}} \nc{\CC}{{\mathbb C}} \nc{\ZZ}{{\mathbb Z}}
\nc{\FF}{{\mathbb F}} \nc{\NN}{{\mathbb N}} \nc{\QQ}{{\mathbb Q}} \nc{\PP}{{\mathbb P}}

\nc{\vs}{\vspace{.2cm}} \nc{\vsp}{\vspace{1cm}} \nc{\ip}{{\langle\cdot,\cdot\rangle}}
\nc{\ipp}{(\cdot,\cdot)} \nc{\la}{\langle} \nc{\ra}{\rangle} \nc{\unm}{\tfrac{1}{2}}
\nc{\unc}{\tfrac{1}{4}} \nc{\und}{\tfrac{1}{16}} \nc{\no}{\vs\noindent}
\nc{\lam}{\Lambda^2(\RR^n)^*\otimes\RR^n} \nc{\tangz}{{\rm T}^{\rm Zar}}
\nc{\lamg}{\Lambda^2\ggo^*\otimes\ggo}
\nc{\nor}{{\sf n}}  \nc{\mum}{/\!\!/} \nc{\kir}{/\!\!/\!\!/}
\nc{\Ri}{\tfrac{4\Ric_{\mu}}{||\mu||^2}} \nc{\ds}{\displaystyle}
\nc{\ben}{\begin{enumerate}} \nc{\een}{\end{enumerate}} \nc{\f}{\frac}
\nc{\lb}{[\cdot,\cdot]} \nc{\isn}{\tfrac{1}{||v||^2}}
\nc{\gkp}{(\ggo=\kg\oplus\pg,\ip)} \nc{\ukh}{(\ug=\kg\oplus\hg,\ip)}
\nc{\tgkp}{(\tilde{\ggo}=\kg\oplus\pg,\ip)}
\nc{\wt}{\widetilde}
\nc{\raw}{\rightarrow} \nc{\lraw}{\longrightarrow} \nc{\hqn}{\mathcal{H}_{q,n}}
\nc{\minimatrix}[4]{\left[\begin{smallmatrix} {#1} & {#2} \\ {#3} & {#4} \end{smallmatrix}\right]}

\nc{\ad}{\operatorname{ad}}  \nc{\Aut}{\operatorname{Aut}}   \nc{\Inn}{\operatorname{Inn}}   \nc{\Lie}{\operatorname{Lie}} \nc{\Ad}{\operatorname{Ad}} \nc{\Der}{\operatorname{Der}} \nc{\rad}{\operatorname{r}} \nc{\kf}{\operatorname{B}}

\nc{\End}{\operatorname{End}} \nc{\rank}{\operatorname{rank}} \nc{\Ker}{\operatorname{Ker}} \nc{\tr}{\operatorname{tr}} \nc{\sym}{\operatorname{sym}} \nc{\diag}{\operatorname{diag}} \nc{\proy}{\operatorname{pr}} \nc{\Adj}{\operatorname{Adj}}

\nc{\Hess}{\operatorname{Hess}}  \nc{\dif}{\operatorname{d}} \nc{\sen}{\operatorname{sen}} \nc{\grad}{\operatorname{grad}} \nc{\Order}{\operatorname{O}} \nc{\divg}{\operatorname{div}}

\nc{\Iso}{\operatorname{I}} \nc{\Diff}{\operatorname{Diff}} \nc{\ricci}{\operatorname{Ric}}  \nc{\Rc}{\operatorname{Rc}} \nc{\Ricci}{\operatorname{Ric}} \nc{\Riem}{\operatorname{Rm}} \nc{\scalar}{\operatorname{sc}} \nc{\scalarm}{\hat{\operatorname{R}}} \nc{\riccim}{\widehat{\operatorname{Ric}}} \nc{\tang}{\operatorname{T}} \nc{\vol}{\operatorname{vol}}

\nc{\mm}{\operatorname{M}} \nc{\CH}{\operatorname{CH}} \nc{\Irr}{\operatorname{Irr}} \nc{\mcc}{\operatorname{mcc}}

\theoremstyle{plain}
\newtheorem{theorem}{Theorem}[section]

\newtheorem{corollary}[theorem]{Corollary}
\newtheorem{lemma}[theorem]{Lemma}

\theoremstyle{definition}

\theoremstyle{remark}
\newtheorem{remark}[theorem]{Remark}

\title[On homogeneous warped product Einstein metrics]{On homogeneous warped product Einstein metrics}

\author{Ramiro A.~Lafuente}

\address{Universidad Nacional de C\'ordoba, FaMAF and CIEM, C\'ordoba, Argentina}
\email{rlafuente@famaf.unc.edu.ar}

\thanks{This research was partially supported by grants from CONICET, FONCYT and SeCyT (Universidad Nacional de C\'ordoba)}

\begin{document}

\begin{abstract}
In this article we study homogeneous warped product Einstein metrics and its connections with homogeneous Ricci solitons. We show that homogeneous $(\lambda,n+m)$-Einstein manifolds (which are the bases of homogeneous warped product Einstein metrics) are one-dimensional extensions of algebraic solitons. This answers a question from a paper of C.~He, P.~Petersen and W.~Wylie, where they prove the converse statement. Our proof is strongly based on their results, but it also makes use of sharp tools from the theory of homogeneous Ricci solitons. As an application, we obtain that any homogeneous warped product Einstein metric with homogeneous base is diffeomorphic to a product of homogeneous Einstein manifolds.
\end{abstract}

\maketitle

\section{Introduction}

%

A \emph{$(\lambda,n+m)$-Einstein} manifold is a complete Riemannian manifold $(M^n, g, w)$ where $w$ is a positive smooth function on $M$ satisfying
\begin{equation}\label{lnm}
    \Hess w = \frac{w}{m} \left( \Ricci - \lambda g \right).
\end{equation}
If $m=1$, the additional assumption $\Delta w = -\lambda w$ is made. These spaces have been previously studied in \cite{CSW, KK, HePtrWyl, HePtrWylClassif, HePtrWylUniq}, among others, sometimes with the additional hypothesis that $\partial M \neq \emptyset$ (which we do not consider in this article). Notice that when $w$ is constant, \eqref{lnm} is simply the Einstein equation. A natural geometric interpretation for this equation is that $(M,g_M,w)$ satisfies \eqref{lnm} for $m>1$ if and only if there is an $(n+m)$-dimensional warped product Einstein metric $g_E$ of the form
\[
g_E = g_M + w^2 g_{F^m}, \qquad \hbox{with } \Ricci_{g_E} = \lambda g_E,
\]
see \cite[Proposition 1.1]{HePtrWylClassif}. $(\lambda,n+m)$-Einstein manifolds are also called \emph{$m$-quasi Einstein} manifolds in the literature, because if one defines $f$ by $e^{-f/m} = w$ then the $(\lambda,n+m)$-Einstein equation becomes
\[
    \Ricci^m_f = \Ricci + \Hess f - \frac{\dd f \otimes \dd f}{m} = \lambda g.
\]
$\Ricci^m_f$ is sometimes called the \emph{$m$-Bakry-Emery-Ricci tensor} and it is a generalization of the notion of Ricci curvature for the smooth metric measure space $(M,g, e^{-f} \dd \vol_g)$. 

In this article, we are interested in the case of noncompact homogeneous manifolds. Noncompact homogeneous Einstein manifolds have been extensively studied in the last two decades, see the survey \cite{cruzchica} and the references therein. Recently, a close link between noncompact Einstein homogeneous manifolds and \emph{algebraic} homogeneous Ricci solitons has been found in \cite{alek} and \cite{HePtrWyl}, generalizing the previous work by J.~Lauret in the case of solvmanifolds (\cite{soliton}). It basically states that an algebraic soliton always admits a one-dimensional Einstein extension. It is also proved in \cite{alek} that any noncompact Einstein homogeneous manifold is the one-dimensional extension of an algebraic soliton, provided it admits a transitive \emph{non-unimodular} group of isometries (recall that all known examples admit such a transitive group). These results imply in particular that the Alekseevskii conjecture (cf. \cite[7.57]{Bss}) is equivalent to its much more general analogous for algebraic solitons.

Since the $(\lambda,n+m)$-Einstein equation can be interpreted as the Einstein equation for the $m$-Bakry-Emery-Ricci tensor, and usually many topological and geometric results for Ricci curvature can be extended to this tensor (see e.g.~\cite{WW} and the references therein), it is expected that there should also be an analogous link between algebraic solitons and homogeneous $(\lambda,n+m)$-Einstein manifolds. Indeed, it was also proved by C.~He, P.~Petersen and W.~Wylie that an algebraic soliton always admits a homogeneous $(\lambda,n+m)$-Einstein one-dimensional extension, which in particular implies that these spaces can be isometrically embedded into a homogeneous Einstein manifold with arbitrary codimension. Our main aim in this article is to prove the converse of this result.

\begin{theorem}\label{main}
Any homogeneous $(\lambda,n+m)$-Einstein manifold is a one-dimensional extension of an algebraic homogeneous Ricci soliton.
\end{theorem}

This answers a question in \cite{HePtrWyl} (see Remark 1.10 from that paper). Using that algebraic solitons admit Einstein one-dimensional extensions we get the following corollary.

\begin{corollary}\label{coromain}
Let $(E,g_E) = (M\times_w F, g_M + w^2 g_F)$ be a homogeneous warped product Einstein metric with a homogeneous base $(M,g_M)$. Then, $M$ admits a homogeneous Einstein metric and thus $E$ is diffeomorphic to a product of homogeneous Einstein manifolds.
\end{corollary}

We observe that all the structural results for algebraic solitons given in \cite[Theorem 4.6]{alek} can now be applied to study homogeneous $(\lambda,n+m)$-Einstein spaces.

The tools used in the proof of Theorem \ref{main} include a structure theorem for homogeneous $(\lambda,n+m)$-Einstein manifolds given in \cite{HePtrWyl} (Theorem \ref{thm51}), an inequality for the Ricci curvature from \cite{alek} that comes from geometric invariant theory (Lemma \ref{pie}), and a technical lemma from \cite{Jbl13b} about derivations and the Ricci curvature of certain homogeneous spaces (Lemma \ref{algebraic}).

The article is organized as follows. In Section \ref{pre} we recall some definitions and results about the Ricci curvature of a homogeneous manifold, one-dimensional extensions of homogeneous manifolds, and algebraic homogeneous Ricci solitons. Then, in Section \ref{proof} we prove Theorem \ref{main}.

\vs \noindent {\it Acknowledgements.} The author would like to thank Jorge Lauret for very fruitful discussions about this problem.

\section{Preliminaries}\label{pre}

In this section we review the basic facts about the Ricci curvature of homogeneous manifolds, as well as some recent techniques related to it that will be used along the proof of Theorem \ref{main}. We will also briefly introduce algebraic solitons and one dimensional extensions of Riemannian homogeneous spaces.

\subsection{The Ricci curvature of a homogeneous manifold}\label{riccicurvature}

Let $(G/K,g)$ be a connected, almost effective Riemannian homogeneous space, and let $\ggo = \kg \oplus \pg$ be the $\Ad(K)$-invariant decomposition for $\ggo = \Lie(G)$ where $\pg$ is the orthogonal complement of $\kg$ with respect to the Killing form $\kf$ of $\ggo$ (which is negative definite on $\kg$). The metric $g$ is thus identified with an $\Ad(K)$-invariant inner product $\ip$ on $\pg \simeq T_p G/K$. According to \cite{Bss}, the Ricci operator (or (1,1)-Ricci tensor) at the point $eK$ is the symmetric map given by
\[
    \Ricci_g = \mm_\pg - 1/2 \kf_\pg - S(\ad_\pg H) \in \End(\pg).
\]
Here, $\kf_\pg$ is the restriction of the Killing form to $\pg\times \pg$, that is, $\langle \kf_\pg X, Y \rangle = \tr \ad X \ad Y$ for $X,Y \in \pg$. The vector $H\in \pg$, which is usually called the \emph{mean curvature vector}, is defined by the formula $\langle H,X \rangle = \tr \ad X$, $X\in \pg$, and $(\ad_\pg H) X = [H,X]_\pg$, where $\lb_\pg$ is the Lie bracket $\lb$ of $\ggo$ restricted to $\pg \times \pg$ and projected onto $\pg$. For an endomorphism $E\in (\pg,\ip)$, we denote by $S(E)$ and $A(E)$ its symmetric and skew-symmetric parts, respectively, that is
\[
    S(E) = \unm (E+E^t), \qquad A(E) = \unm (E-E^t).
\]
Finally, the map $\mm_\pg$ has a complicated formula (see equation (13) in \cite{alek}), but surprisingly it can be associated with the moment map for the natural left action of $\Gl(\pg)$ on $\Lambda^2 \pg^* \otimes \pg$, thus it can be implicitly defined by
\begin{equation}\label{mmdef}
    \tr \mm_\pg E = \unc \langle \pi(E) \lb_\pg , \lb_\pg \rangle, \qquad E\in \End(\pg),
\end{equation}
where $\ip$ here denotes the inner product on $\Lambda^2 \pg^* \otimes \pg$ induced by the inner product $\ip$ on $\pg$, and $\pi$ is the derivative of the aforementioned $\Gl(\pg)$ action (which has the property that $\pi(E)\lb_\pg = 0$ if and only if $E\in \Der(\pg,\lb_\pg)$). In other words, we have that $m(\lb_\pg) = \tfrac4{\|\lb_\pg\|^2} \mm_\pg$, where $m: \Lambda^2 \pg^* \otimes \pg \to \sym(\pg)$ is the moment map for the natural action of $\Gl(\pg)$ on that space; see \cite[\S 2]{alek} for more details on this matter.

In general, a map $T\in \End(\ggo)$ that preserves $\kg$ induces an endomorphism of $\pg$ which we will denote by $T_\pg \in \End(\pg)$. Moreover, if $T\in \Der(\ggo)$ then $T_\pg\in \Der(\pg,\lb_\pg)$.

\begin{lemma}\label{ricderiv}
If $D\in \Der(\ggo)$ preserves $\kg$, then $\tr  \left(\Ricci + S(\ad_\pg H)\right) D_\pg = 0$.
\end{lemma}

\begin{proof}
The $\Ad(G)$-invariance of the Killing form implies that $\kf_\pg D_\pg + D_\pg^t \kf_\pg = 0$. This, together with equation \eqref{mmdef} and the fact that $\pi(D_\pg)\lb_\pg = 0$ imply that $\tr \left( \mm_\pg - \unm \kf_\pg \right) D_\pg =0$, as claimed.
\end{proof}

In \cite{standard}, a very powerful tool to tackle problems related with the Ricci curvature of homogeneous manifolds was introduced. This tool was further developed in \cite{einsteinsolv, solvsolitons, alek}, and we present it here in a brief version; we refer the reader to the Appendix in \cite{alek} for a more gentle presentation of these topics.

Let $\ngo\subseteq \pg$ be the nilradical of $\ggo$ (which is in $\pg$ since it lies in the radical of the Killing form), and let $n = \dim \ngo$. Fixing an appropriate orthonormal basis for $\ngo$ (see Remark \ref{base} below), we can view its Lie bracket $\lb_\ngo$ as an element of the vector space $V = \Lambda^2 \ngo^*  \otimes \ngo$. Recall that there is a natural $\Gl_n(\RR)$-action on this vector space. Coming from deep results in geometric invariant theory, there is a $\Gl_n(\RR)$-invariant stratification of $V$ given by the disjoint union
\[
    V \setminus \{ 0\} = \bigcup_{\beta\in \mathcal{B}} \mathcal{S}_\beta
\]
where the strata $\mathcal{S}_\beta$ are indexed by a finite set of diagonal matrices $\mathcal{B}$ whose trace is $-1$. Thus, if $\lb_\ngo \neq 0$ we can associate to it an operator $\beta$, and this operator has many interesting properties which we summarize in the following theorem.

\begin{theorem}\cite{standard, einsteinsolv}\label{git} Let $(\ngo,\lb_\ngo,\ip_\ngo)$ be a nilpotent Lie algebra endowed with an inner product. There exists an orthonormal basis of $\ngo$ such that the following holds: If $\lb_\ngo \neq 0$, let $\beta\in \End(\ngo)$ be the associated symmetric map, with $\tr \beta = -1$ (i.e.~ $\lb_\ngo \in \mathcal{S}_\beta$). If $\ngo$ is abelian, formally take $\beta = \infty$, and assume by convention that the expression $\beta/\|\beta\|^2$ is $0$ in this case. Then, the symmetric map $\beta/\|\beta\|^2  + I$ satisfies the following properties:
\begin{itemize}
    \item[(i)] $\tr \left( \left(\beta/\|\beta\|^2+ I\right) [D,D^t] \right) \geq 0, \quad \forall D\in \Der(\ngo, \lb_\ngo)$   (equality holds if and only if $\left[ \beta/\|\beta\|^2  + I, D\right] = 0$).
    \item[(ii)] $\beta/\|\beta\|^2  + I$ is positive definite.
    \item[(iii)] $\|\beta\| \leq \| m(\lb_\ngo)\|$   (equality holds if and only if $m(\lb_\ngo)$ is conjugate to $\beta$).
    \item[(iv)] $\tr \left(\beta/\|\beta\|^2  + I\right) D = \tr D, \quad \forall D\in \Der(\ngo, \lb_\ngo)$.
    \item[(v)] $\left\langle \pi\left(\beta/\|\beta\|^2  + I\right) \lb_\ngo, \lb_\ngo \right\rangle \geq 0$, with equality if and only if $\beta/\|\beta\|^2  + I \in \Der(\ngo,\lb_\ngo)$.
\end{itemize}
\end{theorem}

\begin{remark}\label{base}
The precise way in which the orthonormal basis for $\ngo$ has to be chosen is so that the technical condition $\beta_{\lb_\ngo} = \beta$ is satisfied; see \cite[Appendix]{alek}, and recall that the $\Gl_n(\RR)$ action on Lie brackets is simply given by the change of basis.
\end{remark}

We define the operator $E_\beta$ on $\pg = \hg \oplus \ngo$ as follows:
\[
E_\beta =
\left\{
  \begin{array}{ll}
    \minimatrix{0}{}{}{\beta/\|\beta\|^2 +  I }, & \hbox{if } \beta\neq \infty \\
    \\
    \minimatrix{0}{}{}{I}, & \hbox{if } \beta = \infty.
  \end{array}
\right.
\]

The following result is essentially \cite[Lemma 2.3]{alek}, but we are stating it here in a different and slightly more general way, so we provide a proof for completeness.

\begin{lemma}\cite{alek}\label{pie} $\tr \left(\Ricci + S(\ad_\pg H)\right) E_\beta \geq 0$, with equality if and only if $\left[\begin{smallmatrix} 0 & 0\\ 0 & E_\beta \end{smallmatrix}\right] \in \Der(\ggo)$.
\end{lemma}

\begin{proof}
Since $\kf_\pg|_\ngo = 0$, we only have to prove that $\tr \mm_\pg E_\beta \geq 0$. The case $\lb_\ngo \neq 0$ is precisely the content of Lemma 2.3 in \cite{alek}. If $\lb_\ngo = 0$, arguing as in the proof of that lemma we see that
\[
\tr \mm_\pg E_\beta = \tfrac14 |\lambda_1|^2 \geq 0,
\]
where $\lambda_1 : \hg \times \hg \raw \ngo$ is the restriction of $\lb$ to $\hg\times \hg$ projected onto $\ngo$.

Concerning equality, if $\minimatrix{0}{0}{0}{E_\beta} \in \Der(\ggo)$ then $E_\beta \in \Der(\pg,\lb_\pg)$ and thus $\tr \mm_\pg E_\beta = 0$ by \eqref{mmdef}. Conversely, assume that equality holds. In the case $\lb_\ngo\neq 0$ we see from the proof of Lemma 2.3 in \cite{alek} and Theorem \ref{git} that we must have
\[
 \lambda_1 = 0, \quad \beta/\|\beta\|^2 + I \in \Der(\ngo, \lb_\ngo), \quad [\beta/\|\beta\|^2 + I, \ad Y|_\ngo] = 0, \qquad \forall Y \in \hg.
\]
These conditions imply that $E_\beta \in \Der(\pg, \lb_\pg)$. But we also observe that
\[
\tr [\beta, \ad Z |_\ngo] \ad Z |_\ngo = 0, \qquad \forall Z \in \kg,
\]
since $\beta$ is symmetric and $\ad Z |_\ngo$ is skew-symmetric. Hence by Theorem \ref{git}, (i) we deduce that $[E_\beta,\ad \kg|_\pg] = 0$, and this gives us $\minimatrix{0}{0}{0}{E_\beta} \in \Der(\ggo)$.

The case of equality when $\lb_\ngo = 0$ is completely analogous, because we also have that $\lambda_1 = 0$, and conditions $[E_\beta |_\ngo,\ad \hg |_\ngo] = 0$, $E_\beta |_\ngo \in \Der(\ngo,\lb_\ngo)$ and $[E_\beta,\ad \kg|_\pg] = 0$ are trivially satisfied.
\end{proof}

Finally, another interesting property of the Ricci operator of certain Riemannian homogeneous spaces is the following lemma, which is crucial in the proof by M.~Jablonski of the fact that all homogeneous Ricci solitons are algebraic. We notice that it holds in a more general context than homogeneous Ricci solitons. Namely, one only needs an orthogonal semidirect sum decomposition $\ggo = \ug \ltimes \ngo$, where $\ngo$ is the nilradical of $\ggo$.

\begin{lemma}\cite[Lemma 4]{Jbl13b}\label{algebraic}
Let $(G/K,g)$ be a Riemannian homogeneous space such that $\ggo$ admits a semidirect sum decomposition $\ggo = \ug\ltimes \ngo$, where $\ngo$ is the nilradical of $\ggo$, $\ug = \kg \oplus \hg$ is thus a reductive subalgebra, $\pg = \hg \oplus \ngo$ is an $\Ad(K)$-invariant complement for $\kg$ in $\ggo$, and $\hg \perp \ngo$ with respect to the inner product induced by $g$ on $\pg$. Then, for any $Y\in \hg$,
\[
\tr \minimatrix{\ad Y|_\hg}{0}{0}{0} \Ricci = 0.
\]
\end{lemma}

\subsection{One dimensional extension of a Riemannian homogeneous space}\label{onedimext}

Given a Riemannian homogeneous space $(G/K,g)$, a derivation $D\in \Der(\ggo)$ that preserves $\kg$, and $\alpha \in \RR$, its \emph{one-dimensional extension} is defined as the Riemannian homogeneous space $(\tilde{G}/K, \tilde{g})$, where $\tilde{G} = \RR \ltimes G$ is the semidirect product whose Lie algebra is given by $\tilde\ggo = \RR \ltimes \ggo$, and the adjoint action of a distinguished element $\xi \in \RR$ on $\ggo$ is given by the derivation $\alpha D$. In addition, the $\Ad(K)$-invariant inner product induced by $g$ on $\pg$ extends to an $\Ad(K)$-invariant inner product on $\tilde{\pg} = \RR \xi \oplus \pg$, by making that decomposition orthogonal and requiring $\xi$ to be of unit norm. This inner product on $\tilde{\pg}$ defines a $\tilde{G}$-invariant metric $\tilde{g}$ on $\tilde{G}/K$.

This construction has been used in the case of solvable Lie groups in \cite{soliton}, and more recently in \cite{alek} and \cite{HePtrWyl}, to establish a close link between Einstein homogeneous manifolds and algebraic solitons. We refer the reader to \cite[\S 2.2]{HePtrWyl} for the calculations of the curvature of the extension in terms of the curvature of the original manifold and the data $D$, $\alpha$.

\subsection{Algebraic solitons}

Homogeneous Ricci solitons are self-similar solutions to the Ricci flow which are also homogeneous Riemannian manifolds. They are generalizations of homogeneous Einstein metrics, and in fact they share many interesting properties with those metrics. Recently, these spaces have been extensively studied (see for instance \cite{Jbl}, \cite{homRS}, \cite{alek}, \cite{Jbl13b}, \cite{HePtrWyl}, \cite{lowdim}, among many others).

A homogeneous Ricci soliton $(M,g)$ is called and \emph{algebraic soliton} with respect to a transitive group of isometries $G$ if for its presentation as a homogeneous space $(M,g) = (G/K, g)$ we have that the Ricci operator (i.e.~ the 1-1 Ricci tensor) at the point $eK$ satisfies
\[
    \Ricci = cI + D_\pg, \qquad D\in \Der(\ggo)
\]
where $\pg$ denotes an $\Ad(K)$-invariant complement for $\kg = \Lie(K)$ in $\ggo = \Lie(G)$. This algebaric condition is very nice and useful to construct examples, and also to obtain structural properties. Remarkably, by using the structural results obtained in \cite{alek} together with a brand new approach, it was recently shown in \cite{Jbl13b} that all homogeneous Ricci solitons are indeed algebraic with respect to its isometry group.

\section{Proof of Theorem \ref{main}}\label{proof}

We proceed in this section with the proof of our main result. First observe that if the $(\lambda,n+m)$-Einstein manifold is \emph{trivial} (i.e.~if it is Einstein), the theorem follows immediately from \cite[Proposition 6.1 (ii)]{alek}, so we will focus on the \emph{non-trivial} case from now on.

First let us recall the following construction for non-trivial homogeneous $(\lambda,n+m)$-Einstein manifolds, in which our proof is mainly based.

\begin{theorem}\cite[Theorem 5.1]{HePtrWyl}\label{thm51}
Let $(\tilde{M} = \tilde{G}/\tilde{K},\tilde{g})$ be a homogeneous $(\lambda,n+m)$-Einstein space which is not Einstein. Then, $(\tilde{M},\tilde{g})$ is the one-dimensional extension of a homogeneous space $(M = G/K,g)$ with a derivation $D\in \Der(\ggo)$ and $\alpha^2 = \frac1{\tr S(D_\pg) - \lambda m}$, satisfying the following conditions:
\begin{itemize}
    \item[(1)] $\Ricci_g = \lambda I + S(D_\pg) + \frac1{\tr S(D_\pg) - \lambda m} [S(D_\pg), A(D_\pg)],$
    \item[(2)] $\divg S(D_\pg) = 0,$
    \item[(3)] $\tr S(D_\pg)^2 = -\lambda \tr S(D_\pg)$,
\end{itemize}
where $\lambda < 0 $. Moreover, the warping function $w \in \mathcal{C}^\infty(\tilde{M})$ is given by $w(r) = e^{\lambda \alpha r}$, where $r$ is the signed distance function on $\tilde{M}$ to the hypersurface $M$.
\end{theorem}

Our main aim is to prove that the conditions on the Ricci curvature given in that theorem imply that the homogeneous space $(G/K,g)$ is actually an algebraic soliton. To do that, we establish a series of technical lemmas that basically imitate the proof of the main structural results for homogeneous Ricci solitons given in \cite{alek}, and after that we use a result of Jablonski (Lemma \ref{algebraic}) to conclude that the derivation used to build the one-dimensional extension must indeed be normal.


We begin with an analogous of \cite[Lemma 4.2]{alek}. The proof, however, is slightly different in this case, and makes heavy use of the fact that there exists a $(\lambda,n+m)$-Einstein one-dimensional extension.

\begin{lemma}\label{ineqF}
Let $(M=G/K,g)$ be a Riemannian homogeneous space with reductive decomposition $\ggo = \kg \oplus \pg$, mean curvature vector $H\in \pg$, $p = eK \in G/K$, and assume that it admits a one-dimensional extension by a derivation $D\in \Der(\ggo)$ and a constant $\alpha\in \RR$ which is a $(\lambda,n+m)$-Einstein manifold $(\tilde{M} = \tilde{G}/K,\tilde{g})$ with warping function $w(r) = e^{\lambda \alpha r}$, where $r$ is the signed distance function on $\tilde{M}$ to the hypersurface $M$. Also, let $F = S(\ad_\pg H + D_\pg) \in \End(T_p M)$. Then,
\[
\tr F^2 + \lambda \tr F \leq 0,
\]
with equality if and only if $D(H) = 0$.
\end{lemma}

\begin{proof}
We will follow the notation of Section \ref{onedimext}. We have that $\tilde{G} = \RR\ltimes G$, the action of $\RR$ on $G$ is induced by $\alpha$ and $D$, $\tilde{\pg} = \RR \xi \oplus \pg$ is a reductive complement for $\kg$ in $\tilde{\ggo}$, and observe that the mean curvature vector $\tilde{H}$ of the extension $(\tilde{G}/K,\tilde{g})$ is given by
\[
    \tilde{H} = (\alpha \tr S(D_\pg)) \xi + H.
\]
Thus, we have that
\begin{align}\label{SadH}
    S(\ad_{\tilde{\pg}} \tilde{H} ) &= S(\ad_{\tilde{\pg}} H) + (\alpha^2 \tr S(D_\pg))\minimatrix{0}{0}{0}{ S(D_\pg) }\\
    & =  \minimatrix{0}{\ast}{\ast}{S(\ad_\pg H)} + (\alpha^2 \tr S(D_\pg))\minimatrix{0}{0}{0}{ S(D_\pg) }, \nonumber
\end{align}
where the blocks are with respect to the decomposition $\tilde{\pg} = \RR \xi \oplus \pg$. Also, for the adjoint action of $H+\xi \in \tilde\ggo$ restricted to $\tilde\pg$ one has
\begin{equation}\label{SadHxi}
    S(\ad_{\tilde{\pg}}(H + \xi)) = \minimatrix{0}{\ast}{\ast}{F}.
\end{equation}

On the other hand, by the proof of \cite[Theorem 3.3]{HePtrWyl} we have that $\alpha^2 = 1/(\tr S(D_\pg) - \lambda m)$. Using also the formula for $\Hess w$ given in that proof we have that, at the point $p \in \tilde{M}$, the tensor $\Hess w$ is given by
\[
    \frac{m}{w} \Hess w = \langle A \cdot,\cdot\rangle, \quad \hbox{ where } \quad A = m\lambda \alpha^2 \minimatrix{\lambda}{0}{0}{- S(D_\pg)} \in \End(\tilde\pg).
\]
So, the fact that $(\tilde{M},\tilde{g})$ is $(\lambda,n+m)$-Einstein implies that its Ricci operator at $p$ satisfies
\begin{equation*}
    \widetilde{\Ricci} = \lambda I + m\lambda \alpha^2 \minimatrix{\lambda}{0}{0}{- S(D_\pg)},
\end{equation*}
and so by using \eqref{SadHxi} and the formula $\alpha^2 = 1/(\tr S(D_\pg) - \lambda m)$ we get
\begin{equation}\label{Rictilde}
     \widetilde{\Ricci} +  S(\ad_{\tilde\pg} \tilde H) = \lambda I +  \minimatrix{m\lambda^2 \alpha^2}{0}{0}{0} + S(\ad_{\tilde{\pg}}(H + \xi)).
\end{equation}
Finally, putting together equations \eqref{SadHxi}, \eqref{Rictilde} and using Lemma \ref{ricderiv} with the derivation $\ad (H+\xi) \in \Der(\tilde\ggo)$  we obtain
\begin{align*}
    0 \quad =& \quad \tr \left(\widetilde{\Ricci} + S(\ad_{\tilde\pg} \tilde H) \right) S(\ad_{\tilde\pg} ( H + \xi)) \\
    =&\quad \lambda \tr S(\ad_{\tilde\pg} ( H + \xi)) + \tr S(\ad_{\tilde\pg} ( H + \xi))^2 \\
    \geq&\quad \lambda \tr F + \tr F^2.
\end{align*}
It is clear that we have equality if and only if $S(\ad_{\tilde\pg} H) = \minimatrix{0}{0}{0}{S(\ad_\pg H)}$, which is in turn equivalent to saying that $[H,\xi] = 0$, or $D(H) = 0$.
\end{proof}

With this inequality for $F$ and the condition on the Ricci curvature given in \cite[Theorem 5.1]{HePtrWyl}, we see in the following lemmas how this spaces resemble algebraic solitons.

\begin{lemma}\label{cuenta}
Let $(G/K,g)$ be a Riemannian homogeneous space with reductive decomposition $\ggo = \kg \oplus \pg$, with $\pg = \hg \oplus  \ngo$ as in Section \ref{riccicurvature}, and assume that its Ricci operator at the point $p = eK\in G/K$ satisfies
\begin{equation}\label{conditionRic}
\Ricci_g = \lambda I + S(D_\pg)+ \alpha^2 [S(D_\pg),A(D_\pg)],
\end{equation}
where $\alpha, \lambda\in \RR$, $\lambda<0$ and $D\in \Der(\ggo)$. Furthermore, assume that $\tr F^2 + \lambda \tr F \leq 0,$ where $F = S(D_\pg + \ad_\pg H)$. Then,
\begin{itemize}
    \item[(i)] $[\hg,\hg]_\pg \subseteq \hg$, or equivalently, $\ug$ is a Lie subalgebra of $\ggo$.
    \item[(ii)] $\beta/\|\beta\|^2 + I \in \Der(\ngo)$ and $[E_\beta, \ad \hg|_\pg] = [E_\beta, D_\pg]= 0$.
    \item[(iii)] $S(D_\pg + \ad_\pg H) = -\lambda E_\beta$.
    \item[(iv)] $\ad_\pg H + D_\pg$ is a normal operator.
\end{itemize}
\end{lemma}


\begin{proof}
Let us rewrite the condition on $\Ricci_g$ as
\[
\unm \alpha^2 [D_\pg, D_\pg^t] + \Ricci_g + S(\ad_\pg H) = \lambda I + F.
\]
By taking traces against $E_\beta$ and using Lemma \ref{pie} and Theorem \ref{git}, (i) we obtain
\begin{equation}\label{F}
\tr (\lambda I + F) E_\beta \geq 0.
\end{equation}
Using the previous inequality, together with $\tr F^2 + \lambda \tr F \leq 0$, $\tr F E_\beta = \tr F$ (which follows from \cite[Lemma 2.6]{alek} and Theorem \ref{git}, (iv)) and $\tr E_\beta^2 =  \tr E_\beta$, we obtain that
\begin{align*}
    \left( \tr F E_\beta\right)^2 \geq& \, \,(-\lambda \tr E_\beta) \tr F = \tr E_\beta^2 (-\lambda \tr F) \\ \geq& \,\, \tr E_\beta^2 \tr F^2,
\end{align*}
which is a ``reverse'' Cauchy-Schwartz inequality. Thus, we must have equality everywhere, and conditions (i)-(iii) are easily obtained as in the proof of \cite[Proposition 4.1]{alek}. Condition (iv) follows immediately from (ii) and (iii).
\end{proof}

The previous result implies that the operators $\ad_\pg H$ and $S(D_\pg)$ have the following forms with respect to the orthogonal decomposition $\pg = \hg \oplus \ngo$:
\begin{equation}\label{formadH}
    \ad_\pg H = \minimatrix{\ad H|_\hg}{0}{0}{\ad H|_\ngo}, \qquad S(D_\pg) = \minimatrix{-S(\ad H|_\hg)}{0}{0}{S(D_\ngo)},
\end{equation}
with $S(\ad H|_\ngo) + S(D_\ngo) = -\lambda(\beta/\|\beta\|^2 +  I)$.

\begin{lemma}\label{lemmaadytder}
For any $Y\in \hg$ we have that $(\ad Y|_\ngo)^t \in \Der(\ngo)$, and $D_\ngo^t \in \Der(\ngo)$. Moreover, $\sum [(\ad Y_i)|_\ngo, (\ad Y_i)|_\ngo^t] + \alpha^2[D_\ngo, D_\ngo^t] = 0$ and $\mm_\ngo = \lambda I + F|_\ngo$, where $\mm_\ngo$ is the map defined in \eqref{mmdef} corresponding to the Lie algebra $(\ngo,\lb_\ngo,\ip|_\ngo)$.
\end{lemma}

\begin{proof}
The proof of the first two claims is analogous to that of \cite[Lemma 4.9]{alek}. The only difference here is that equation (36) in that paper translates to
\begin{equation}\label{adytder}
\mm_\ngo + \unm \sum [(\ad Y_i)|_\ngo, (\ad Y_i)|_\ngo^t] + \unm \alpha^2[D_\ngo, D_\ngo^t] + \tfrac{\lambda}{\|\beta\|^2} \beta = 0.
\end{equation}
Notice that the term $\unm \alpha^2[D_\ngo, D_\ngo^t]$ can be treated as if it were one of the form $[(\ad Y_i)|_\ngo, (\ad Y_i)|_\ngo^t]$.

To prove the last assertion, one argues as in the proof of \cite[Theorem 4.6]{alek}, using that in equation \eqref{adytder} the maps $\unm \sum [(\ad Y_i)|_\ngo, (\ad Y_i)|_\ngo^t] + \unm \alpha^2[D_\ngo, D_\ngo^t]$ and $\mm_\ngo + \tfrac{\lambda}{\|\beta\|^2} \beta$ are mutually orthogonal by the first part of the proof.
\end{proof}

\begin{lemma}
The operator $S(\ad_\pg H)$ satisfies $S(\ad H|_\hg) = 0$.
\end{lemma}

\begin{proof}
Lemma \ref{cuenta}, (i) allows us to use Lemma \ref{algebraic}, which we will use for $Y = H \in \hg$. The fact that the operators $\minimatrix{S(\ad H|_\hg)}{0}{0}{0}$ and $S(D_\pg)$ commute (which follows from \eqref{formadH}), together with Lemma \ref{algebraic}, and \eqref{conditionRic} imply that
\begin{align*}
    0 =& \tr \minimatrix{S(\ad H|_\hg)}{0}{0}{0} \Ricci = \lambda \tr S(\ad H|_\hg) - \tr S(\ad H|_\hg)^2 \\
    & + \alpha^2 \tr \minimatrix{S(\ad H|_\hg)}{0}{0}{0} \left[S(D_\pg), A(D_\pg)\right] \\
     =&  - \tr S(\ad H|_\hg)^2 + \tr \left[\minimatrix{S(\ad H|_\hg)}{0}{0}{0}, S(D_\pg)\right] A(D_\pg) \\
    =& - \tr S(\ad H|_\hg)^2,
\end{align*}
thus $S(\ad H|_\hg) = 0$ as claimed. Notice that in the second equality we are using that $\tr\ad H|_\hg = 0$, which follows from \cite[Lemma 2.6]{alek}.
\end{proof}

To conclude the proof of the theorem, let us show that condition $S(\ad H|_\hg) = 0$ implies that $\ad_\pg H$ and $D_\pg$ are both normal operators. This will follow in a way analogous to the proof of \cite[Proposition 4.14]{alek}. Indeed, now we have that
\[
S(\ad_\pg H) = \minimatrix{0}{0}{0}{S(\ad H|_\ngo)}.
\]
On the other hand, using Lemma \ref{lemmaadytder} and \cite[Lemma 4.4]{alek} we have that $\mm_\pg|_\ngo = \mm_\ngo$. Therefore,
\begin{align*}
    \tfrac14 \left\|\pi\left( \left( \ad_\pg H\right)^t \right) \lb_\pg \right\|^2 =& \tr M_\pg [\ad_\pg H, \left( \ad_\pg H\right)^t ] \\
    =& \tr M_\ngo [\ad{H}|_\ngo, \left( \ad{H}|_\ngo\right)^t ] \\
    =& \tfrac14 \left\|\pi\left( \left( \ad{H}|_\ngo\right)^t \right) \lb_\ngo\right\|^2 = 0,
\end{align*}
and thus $(\ad_\pg H)^t \in \Der(\lb_\pg)$. Besides, we have that $(\ad_\pg H)^t H = 0$, since
\[
    \langle (\ad_\pg H)^t H , X \rangle = \langle H, [H,X] \rangle = \tr [\ad H, \ad X] = 0, \quad \forall X\in \pg.
\]
Hence,
\[
    \left( \ad_\pg H\right)^t \left( [H,X] \right) = [\left( \ad_\pg H\right)^t H, X] + [H, \left( \ad_\pg H\right)^t X] = [H, \left( \ad_\pg H\right)^t X],
\]
and this says that $\ad_\pg H$ is normal.

Also, since we have equality in Lemma \ref{ineqF} we obtain $D(H) = 0$, so
\[
    [D, \ad H] = \ad (D(H)) = 0,
\]
and in particular $[D_\pg, \ad_\pg H] = 0$. Finally, using this fact together with Lemma \ref{cuenta}, (iv) and the fact that $\ad_\pg H$ is normal, we can conclude that $D_\pg$ is normal, which implies that $(G/K,g)$ is an algebraic soliton.

\bibliography{ramlaf}
\bibliographystyle{amsalpha}

\end{document}